\documentclass[12pt,reqno]{amsart}
\makeatletter
\newcommand\useleftqno{\renewcommand\@eqnnum{\hb@xt@.01\p@{}%
		\flap{\normalfont\normalcolor
			\hskip -\displaywidth(\theequation)}}}
\usepackage{amsmath,amssymb,amsbsy,amsfonts,latexsym,amsopn,amstext,cite,
	amsxtra,euscript,amscd,bm}
\usepackage{amsmath,amssymb,amsbsy,amsfonts,amsthm,latexsym,
	amsopn,amstext,amsxtra,euscript,amscd,mathrsfs}
\usepackage{mathtools}
\usepackage{todonotes}
\usepackage{url}
\usepackage{empheq}
\usepackage[colorlinks,linkcolor=blue,anchorcolor=blue]{hyperref}

\makeindex \makeatletter
\def\captionof#1#2{{\def\@captype{#1}#2}}
\makeatother

\newcounter{tablegroup}
\newcounter{subtable}[tablegroup]

\newcommand{\tend}[3][]{\xrightarrow[#2\to#3]{#1}}

\newcommand{\ds}{\displaystyle}

\newtheorem{thm}{Theorem}[section]

\newtheorem{lem}[thm]{Lemma}

\newtheorem{Que}[thm]{Question}
\newtheorem*{Conj}{\textbf{Conjecture}}

\numberwithin{equation}{section}

\newcommand{\N}{\mathbb N}

\newcommand{\C}{\mathbb C}
\newcommand{\bmu}{\bm \mu}

\newcommand{\bml}{\bm \lambda}
\newcommand{\Rea}{\textrm{Re}}

\begin{document}
	\title[ Exponential sums of the
	M\"{o}bius function .]
	{Exponential sums of the
		M\"{o}bius function and flat polynomials}
	
	\author[\MakeLowercase{el}. H. \MakeLowercase{el} Abdalaoui]{\MakeLowercase{el} Houcein 
		\MakeLowercase{el} Abdalaoui}
	
	\address{e. H. el Abdalaoui, Normandy University of Rouen,
		Department of Mathematics, LMRS  UMR 6085 CNRS, Avenue de l'Universit\'e, BP.12,
		76801 Saint Etienne du Rouvray - France.}
	\email{elhoucein.elabdalaoui@univ-rouen.fr}
	
	\subjclass[2000]{ 11L03,11M26, 37A30, 30A40}

	\keywords{Erd\"{o}s flat polynomials, semi-flat polyomials, Liouville function, M\"{o}bius function, Riemann hypothesis, Littlewood criterion}
	
	\begin{abstract}
		It is shown that if the analytic polynomials with M\"{o}bius  are $L^p$-semi-flat then the Riemann hypothesis holds. It turns out that this problem is a particular case of the weak form of flat polynomials problem asked by Erd\"os in his 1957's paper. Under an extra-condition, we establish also the converse. We further point out an oversight in Littlewood paper \cite{Littlewood} and we correct it. It is also shown that the recent flat polynomials in Littlewood sense constructed by P. Balister and al. are not $L^\alpha$-flat, for any $\alpha \geq 0$. 
	\end{abstract}
	\maketitle
	\section{Introduction}
	The purpose of this article is to present an elementary question on the norms of some polynomials with coefficients $\pm 1$ on the circle which implies Riemann hypothesis. This question can be seen as a special case in the weak form of the Erd\"{o}s-Newman problem on the existence of the so called flat polynomials. \\
	
	Let $(P_n(z))$ be a sequence of analytic polynomials on the circle, that is, for any $n \geq 1$,
	
	$$P_n(z)=\sum_{j=0}^{d_n}a_{j,n}z^j, \; \; |z|=1. $$
	
	This sequence is said to be  $L^2$- normalized polynomials if $$\big\|P_n\big\|_2=\sqrt{\sum_{j=0}^{d_n}|a_{j,n}|^2}=1.$$ 
	If further it converge to 1 in some sense then it is said to be flat, and it is ultraflat if the convergence is uniform.\\

	We warn the reader that this notion should not be confused with the notion of flatness introduced by Littlewood in his 1966's  paper \cite{Littlewood}. For this later notion (which we attribute to Littlewood), the sequence of $L^2$- normalized polynomials $(P_n)$ is flat in the sense of Littlewood if there is a constant $c,C>0$ such that $c \leq |P_n(z)|  \leq C,$ for all $z$ with modulus $1$. Therein, Littlewood asked if there exist a sequence of flat polynomials with coefficients of modulus $1$ or $\pm 1$ which are flat in his sense. He further stated three conjectures on flatness \cite{Littlewood} and \cite[Problem 19]{Littlewood2}. Very recently, using Rudin-Shapiro polynomials combined with Spencer's six deviations lemma,  P. Balister and al. constructed a flat polynomials in the Littlewood sense  \cite{Bal-al}. However, those polynomials are not $L^\alpha$-flat, for any $\alpha \geq 0$ \cite{elabdal-UDT}. We further notice  that Littlewood addressed Erd\"{o}s-Newman question in several papers \cite{Littlewood3}, \cite{Littlewood4}, \cite{Littlewood5}  . Therein, he gives an alternative proof that the cosines polynomials with coefficients $\pm 1$ are not ultraflat. This later result is due to Erd\"{o}s. In fact, Littlewood established that those polynomials are not $L^\alpha$-flat, for any $\alpha \geq 0.$ He also provided a condition on the coefficients of a real trigonometric polynomials to insure that those polynomials are not $L^\alpha$-flat. But, it is seems that Littlewood had conflicting feelings about the existence of ultraflat polynomials. 
	 He further stated in \cite[p.334]{Littlewood4} that the sequence of trigonometric polynomials 
	$$P_n(\theta)=\sum_{j=1}^{n}\cos(2^{2^j}\theta),$$
	is $L^\alpha$-flat for any $\alpha>0$. But, according the central limit theorem for the Hadamard lacunary trigonometric series due to Salem-Zygmund \cite[Vol 2, p. 264]{Zygmund} combined with the methods in \cite[p.176]{Abd-Nad}, it follows that for any $1 \leq \alpha <2$, $$\big\|P_n\big\|_\alpha \tend{n}{+\infty}\Gamma\Big(\frac{\alpha}{2}+1\Big).$$ 

    Here, we are interest in the semi-flatness of an analytic trigonometric polynomials with Liouville  or M\"{o}bius coefficients
    (see \eqref{eqLiouville} and \eqref{eqMobius}). A sequence of $L^2$- normalized polynomials $(P_n)$ is semi-flat if $\big|P_n(z)\big|$ is bounded above by some absolutely constant in some sense. Precisely, we consider $L^\alpha$-semi-flat polynomials with $\alpha>2$, that is, a sequence of analytic trigonometric polynomials $(P_n)$ with coefficients in $\{\pm 1, 0\}$ such that 
     $\big\|P_n(z)\big\|_{\alpha}$ is bounded above by some absolutely constant which may depend on $\alpha$.\\

    Following \cite{MS}, the motivation of the study of the polynomials $(P_n(\theta))$ with Liouville  or M\"{o}bius coefficients goes back to P\'{o}lya's approach to the Riemann hypothesis. Therein, the authors noticed that although Haselgrove \cite{H} disproved P\'{o}lya's conjecture, there is a good reason to resurrect this approach. It turns out that, Hajela and Smith \cite{HS}, under Generalized Riemann Hypothesis (GHR) established the $\Big\|P_n(\theta)\Big\|_\infty= O(n^{\frac{5}{6}+\epsilon})$. Later, under the same hypothesis, Baker and Harman \cite{BH} improved the exponent to $\frac{3}{4}$. Unconditionally, Murty and Sankaranarayanan\cite{MS}, proved that for $\theta$ of type 1, $|P_n(\theta)| \ll n^{\frac{4}{5}+\epsilon}$. Maier and Sankaranarayanan \cite{MaS} improved the exponent to $\frac{3}{4}$, but under some extra-condition. For Weyl type polynomials with Liouville  or M\"{o}bius coefficients, Y. Jiang and G. L\"{u} established recently similar results, we refer to \cite{JL} and the references therein.

    Our result is in the spirit of the previous investigations. It is also our hope that it may shed some light on the connection between the distribution of those polynomials and Riemann hypothesis.\\

   The plan of the paper is as follows. In Section \ref{Tools}, we recall a basic notions and the results that we need in the sequel. In section \ref{main}, we state and prove our main result. 
\section{Set-up and tools}\label{Tools}

Let us denote by the circle $S^1=\big\{z \in \C, |z|=1\big\}$ and for any $\alpha>1$, we define the $L^\alpha$-norm of any trigonometric polynomials $P$ by  
$$\big\|P\big\|_\alpha=\Big(\int \big|P(z)\big|^\alpha dz\Big)^{\frac{1}{\alpha}},$$
where $dz$ is the Lebesgue measure on the circle. For any $n \in \N^*$, we denote by 
$(\xi_{n,j})_{j=0}^{n-1}$, the  the $n$th roots of unity given by 
$$ \xi_{n,j}=e^{\frac{2 i \pi j}{n}},~~~~j=0,\cdots, n-1.$$
The sequence of analytic trigonometric polynomials $(P_n)$ is said to be $L^\alpha$-semi-flat polynomials with $\alpha \geq 1$ if there exist a constant $K_\alpha>0$ such that for any $n \in \N$, 

$$\big\|P_n(z)\big\|_{\alpha} \leq K_\alpha.$$ 

The analytic trigonometric polynomials with Liouville  or M\"{o}bius coefficients are defined as follows
\begin{align}
P_n(z)&=\sum_{j=0}^{n-1}\bml(j+1)z^j,~~~|z|=1, \label{eqLiouville}\\
Q_n(z)&=\sum_{j=0}^{n-1}\bmu(j+1)z^j,~~~|z|=1 \label{eqMobius}
\end{align}
where $\bml$ and $\bmu$ are respectively the Liouville function and the M\"{o}bius function. The  Liouvile function  is given by 
\begin{equation*}\label{Liouville}
\bml(n)= \begin{cases}
1 {\rm {~if~}} n=1; \\
(-1)^r  {\rm {~if~}} n
{\rm {~is~the~product~of~}} r {\rm {~not~necessarily~distinct~prime~numbers}}; 
\end{cases}
\end{equation*}
The Liouville function is related to another famous functions in number theory called the M\"{o}bius function. Indeed, 
the M\"{o}bius function is defined for the positive integers $n$ by
\begin{equation*}\label{Mobius}
\bmu(n)= \begin{cases}
\bml(n) {\rm {~if~}} n \rm{~is~not~divisible~by~the~square~of~any~ prime}; \\
0  {\rm {~if~not}}
\end{cases}
\end{equation*}
Those two functions are related to 
Riemann $\zeta$-function via the formulae
\[
\sum_{n=1}^{+\infty}\frac{\bmu(n)}{n^s}=\frac{1}{\zeta(s)}, \qquad
\sum_{n=1}^{+\infty}\frac{\bml(n)}{n^s}=\frac{\zeta(2s)}{\zeta(s)}
~~{\rm {with}}~~~ \Rea{(s)}>1,
\]
and 
\[
\sum_{n=1}^{+\infty}\frac{|\bmu(n)|}{n^s}=\frac{\zeta(s)}{\zeta(2s)}
~~{\rm {with}}~~~ \Rea{(s)}>1.
\]
Let us further notice that the Dirichlet inverse of the Liouville function is the absolute value of the M\"{o}bius function.\\

For the reader's convenience, we briefly recall some useful well-known results on the Riemann $\zeta$-function. The  Riemann $\zeta$-function is defined, 
for $s \in \C$, $\Rea(s)>1$ by 
$$\zeta(s)=\sum_{n=1}^{+\infty}\frac{1}{n^s},$$
or by the Euler formula
$$\zeta(s)=\prod_{\overset{p}{\textrm{~prime}}}\Big(1-\frac{1}{p^s}\Big)^{-1}.$$
It is easy to check that $\zeta$ is analytic for $\Rea(s)>1$. Moreover, it is well-known that $\zeta$ is regular for all values of $s$ except $s=1$, where there is a
simple pole with residue 1. Thanks to the functional equation \footnote{See \cite[Chap. II]{Titchmarsh} for its seven proofs.} 
$$\zeta(s)=2^{s}\pi^{s-1}\sin\Big(\frac{\pi s}{2}\Big) \Gamma(1-s)\zeta(1-s),$$
where $\Gamma$ is the gamma function given by 
$$\Gamma(z)=\int_{0}^{+\infty}x^{z-1}e^{-x}dx,~~~~~~~~~\Rea(z)>0.$$ We notice that the gamma function never vanishes and it is analytic 
everywhere except at $z=0, -1, -2, ...,$ with the residue at $z=-k$ is equal to
$\frac{(-1)^k}{k!}$.   We further have the following  formula (useful in the proof of the functional equation) 
$$\Gamma(s) \sin\Big(\frac{\pi s}{2}\Big)=\int_{0}^{+\infty} y^{s-1} \sin\big(y\big) dy, 
~~~~0<\Rea(s)<1.$$
For the proof of it we refer to \cite[p.82]{Redmancher}. Changing $s$ to $1-s$, we obtain
$$\zeta(1-s)=2^{1-s}\pi^{-s}\cos\Big(\frac{\pi s}{2}\Big) \Gamma(s)\zeta(s).$$
Putting
$$\xi(s)=\frac{s(s-1)}{2}\pi^{\frac{-s}{2}}\Gamma\big(\frac{s}{2}\big)\zeta(s),$$
and 
$$E(s)=\xi\Big(\frac{1}{2}+is\Big).$$
It follows that
$$\xi(s)=\xi(1-s),\quad \quad \textrm{and} \quad \quad E(z)=E(-z).$$  
We further recall that 
$$\zeta(s)\Gamma(s)=\int_{0}^{+\infty}\frac{x^{s-1}}{e^x-1} dx,\quad \quad \Rea(s)>1.$$
Therefore, it is easy to check that $\zeta$ has no zeros for $\Rea(s)>1$. It follows also from the functional equation that $\zeta$ has no zeros for $\Rea(s)<0$ except for simple zeros at 
$s=-2,-4, \cdots$. Indeed, $\zeta(1-s)$ has no zeros for $\Rea(s)<0$, $\sin\big(\frac{s\pi}{2}\big)$ has simple zeros at $s=-2,-4,\cdots$.  
It is also a simple matter to see that 
$\xi(s)$ has no zeros for $\Rea(s)>1$ or $\Rea(s)<0$. Hence its  zeros which are also the zeros of $\zeta$ lie in the 
strip $0 \leq  \Rea(s) \leq 1$. Hardy point out that for $ \Rea(s)>1$,  it is easily seen that 

\begin{align*}
	\sum_{n=1}^{+\infty}\frac{(-1)^{n-1}}{n^s}=&\sum_{n=1}^{+\infty}\frac{1}{n^s}-2\sum_{n=1}^{+\infty}\frac{1}{2^s n^s} \\
	=&\big(1-2^{1-s}\big)\zeta(s),
\end{align*}

This formula allows us to continue $\zeta$ analytically to half-plan $\Rea(s)>0$ with simple pole at $s=1$. 
We further have $\zeta(s)\neq 0$ for all $s>0$ since 
$ \sum_{n=1}^{+\infty}\frac{(-1)^{n-1}}{n^s}>0.$\\

We thus conclude that all zeros of $\zeta$ are complex. 
The functional equation allows us also to see that if $z$ is a zero then $1-z$ and 
$1-\overline{z}$ are also a zeros. Whence, the zeros of $\zeta$ lie on the vertical line $\Rea(s)=\frac{1}{2}$ or occur in pairs symmetrical about this line.\\

\begin{Conj}[Riemann hypothesis (RH)]
All nontrivial zeros of $\zeta$ lie on the critical line  $\Rea(s)=\frac{1}{2}$.
\end{Conj}
 
Here, we need the following characterization of Riemann  hypothesis due to Littlewood \cite{Littlewood-cras}.

\begin{lem}\label{Littlewood} The Riemann hypothesis is equivalent to 
	\begin{eqnarray}\label{RH2}
	\left|\ds \sum_{n=1}^{x}\bml(n)\right|=o\left(x^{\frac12+\varepsilon}\right)\qquad
	{\rm as} \quad  x \longrightarrow +\infty,\quad \forall \varepsilon >0
	\end{eqnarray}
\end{lem}
For the proof of Littlewood, we refer also to \cite[p.371]{Titchmarsh}, \cite[p.261]{Edwards}. Notice that by applying Bateman-Chowla trick, the same conclusion can be drawn for $\bmu$.\\

We will further need the following fundamental inequalities from the interpolation theory due to Marcinkiewz \& Zygmund \cite[Theorem 7.10, Chapter X, p.30]{Zygmund}.
\begin{lem}\label{MZ}
	For
	$\alpha > 1$, $n \geq 1$, and any analytic trigonometric polynomial $P$ of degree $\leq n-1$,
	\begin{eqnarray}\label{MZ1}
	\frac{A_{\alpha}}{n}\sum_{j=0}^{n-1}\big|P(\xi_{n,j})\big|^{\alpha}
	\leq \int_{S^1}\Big|P(z)\Big|^{\alpha} dz \leq \frac{B_{\alpha}}{n}\sum_{j=0}^{n-1}\big|P(\xi_{n,j})\big|^{\alpha},
	\end{eqnarray}
	where  $A_{\alpha}$ and $B_{\alpha}$ are independent of $n$ and $P$.
\end{lem}
For the trigonometric polynomials, Marcinkiewz-Zygmund interpolation inequalities can be stated as follows \cite[Theorem 7.5, Chapter X, p.28]{Zygmund}.
\begin{lem}\label{MZ2}
	For $\alpha > 1$, $n \geq 1$, and any trigonometric polynomial $P$ of degree $\leq n$,
	\begin{eqnarray}
	  \frac{A_{\alpha}}{2n+1}\sum_{j=0}^{2n}\big|P(\xi_{2n+1,j})\big|^{\alpha}
	\leq \int_{S^1}\Big|P(z)\Big|^{\alpha} dz \leq \frac{B_{\alpha}}{2n+1}\sum_{j=0}^{2n}\big|P(\xi_{2n+1,j})\big|^{\alpha},
\end{eqnarray}
	where  $A_{\alpha}$ and $B_{\alpha}$ are independent of $n$ and $P$.
\end{lem}
\section{main result and its proof}\label{main}
In this section, we start by stating our main result.
\begin{thm}\label{thmain}
		Let $\alpha>2$ and suppose that the sequence of analytic polynomials $(P_n)$ is $L^\alpha$-semi-flat, that is, for each $n \in \N^*$, 
		$$\Big\|\frac{1}{\sqrt{n}} \sum_{j=0}^{n-1}\bml(j+1) z^j\Big\|_\alpha < C_\alpha,$$
		for some constant $C_\alpha$.  Then the Riemann hypothesis holds.
\end{thm}
We are not able to see that the converse of our main theorem holds. although, we have the following
\begin{thm}\label{cvthmain} Suppose the Riemann hypothesis holds and for any $\alpha>2$, the sequence $$\Bigg( \frac{1}{n^{1+\frac{\alpha}{2}}} \sum_{k=0}^{n-1} \Big|\sum_{j=0}^{n-1}\bml(j+1)\xi_{n,k}^j \Big|^\alpha\Bigg)_{n \geq 1}$$ is bounded.   
	Then the sequence of analytic polynomials $(P_n)$ is $L^\alpha$-semi-flat.
\end{thm}
\begin{proof}[\textbf{of the main theorem.}] Assume that the Riemann Hypothesis does not holds. Then, according to Littlewood criterion (Lemma \ref{Littlewood}), there exist $c>0$ and $\epsilon>0$ such that for infinitely many positive integers $n$, we have
	$\Big|\ds \sum_{j=1}^{n}\bml(j)\Big| \geq c. n^{\frac{1}{2}+\epsilon}.$ Let $\alpha>1$ such that $\alpha \epsilon>1$. Then, by Marcinkiewicz-Zygmund inequalities,
	we have
	\begin{eqnarray*}
		\Big\|\frac{1}{\sqrt{n}} \sum_{j=0}^{n-1}\bml(j+1) z^j\Big\|_\alpha^\alpha &\geq& A_\alpha 
		\frac{1}{n^{\frac{\alpha}{2}+1}}\sum_{k=0}^{n-1}
		{\Big|\ds \sum_{j=0}^{n-1}\bml(j+1)\xi_{n,k}^j\Big|^\alpha}\\
		&\geq&
		A_\alpha 
		\frac{1}{n^{\frac{\alpha}{2}+1}}
		{\Big|\ds \sum_{j=0}^{n-1}\bml(j+1)\Big|^\alpha}
	\end{eqnarray*}
But 
\begin{align}
\frac{\Big|\ds \sum_{j=0}^{n-1}\bml(j+1)\Big|^\alpha}{n^{\frac{\alpha}{2}+1}}
=\frac{\Big|\ds \sum_{j=1}^{n}\bml(j)\Big|^\alpha}{n^{\frac{\alpha}{2}+1}}.
\end{align}
Therefore
\begin{eqnarray*}
	\Big\|\frac{1}{\sqrt{n}} \sum_{j=0}^{n-1}\bml(j+1) z^j\Big\|_\alpha^\alpha &\geq&  C_\alpha.\frac{n^{\frac{\alpha}{2}+\alpha\varepsilon}}{n^{\frac{\alpha}{2}+1}}=n^{\alpha\varepsilon-1},
\end{eqnarray*}
	Letting $n \longrightarrow +\infty$, we conclude that 
	$$\Big\|\frac{1}{\sqrt{n}} \sum_{j=0}^{n-1}\bml(j+1) z^j\Big\|_\alpha \tend{n}{+\infty} +\infty.$$
	  
	 This accomplishes the proof of the theorem.    
\end{proof}\\

Let us point out that we can also give a direct proof. Indeed,
assume that for any $\alpha \geq 1$, we have
$$\Big\|\frac{1}{\sqrt{n}} \sum_{j=0}^{n-1}\bml(j+1) z^j\Big\|_\alpha <+\infty.$$
Then, by Marcinkiewicz-Zygmund inequalities, for any $\alpha>1$, there exist $A_\alpha$ such that 
$$A_\alpha 
\frac{\Big|\ds \sum_{j=1}^{n}\bml(j)\Big|^\alpha}{n^{\frac{\alpha}{2}+1}} \leq c_\alpha,$$
where $c_\alpha$ is some positive constant. This gives 
$$\Big|\ds \sum_{j=1}^{n}\bml(j)\Big| \leq C_\alpha n^{\frac{1}{2}+\frac{1}{\alpha}}.$$
Since $\alpha$ is arbitrary, it follows, with the help of \eqref{RH2}, that RH holds.\\
Now, let us give the proof of Theoreom \ref{cvthmain}.\\
\begin{proof}[of Theorem \ref{cvthmain}.] Under our assumptions combined with  Marcinkiewicz-Zygmund Theorem, it is easy to see that the sequence of polynomials $(P_n)_{n \geq 1}$ is $L^\alpha$-semiflat  for any $\alpha \geq 0$. 
\end{proof}
\vskip 0.5cm 
At this point we ask the following
\begin{Que} Does the RH implies the semi-flatness. In another word, is the converse of Theorem \ref{thmain} true?
\end{Que} 

\textbf{Conflict of interest statement.} None.\\

\textbf{Funding.} This research did not receive any specific grant from funding agencies in the public, commercial, or not-for-profit sectors.\\

\textbf{Acknowledgments.} The author would like to thanks Sukumar Das \linebreak Adhikari,  Ramachandran Balasubramanian, Igor Shparlinski,  Jean-Paul Thouvenot, and  Christoph Aistleitner for a stimulating conversations on the subject and their sustained interest and encouragement in this work.  He would like also to thanks the Ramakrishna Mission Vivekananda Educational and Research Institute and the organizers of the ``Workshop on Topological Dynamics, Number Theory and related areas” for the invitation.

\end{document}